\documentclass[11pt]{amsart}
\usepackage[active]{srcltx}
\usepackage{calc,amssymb,amsthm,amsmath,amscd, ulem, mathtools, mathptmx}
\usepackage{alltt}
\usepackage[left=1.35in,top=1.25in,right=1.35in,bottom=1.25in]{geometry}
\RequirePackage[dvipsnames,usenames]{color}

\DeclareFontFamily{OMS}{rsfs}{\skewchar\font'60}
\DeclareFontShape{OMS}{rsfs}{m}{n}{<-5>rsfs5 <5-7>rsfs7 <7->rsfs10 }{}
\DeclareSymbolFont{rsfs}{OMS}{rsfs}{m}{n}
\DeclareSymbolFontAlphabet{\scr}{rsfs}

\usepackage[all,cmtip]{xy}

\newtheorem{theorem}{Theorem}[section]
\newtheorem{lemma}[theorem]{Lemma}
\newtheorem{proposition}[theorem]{Proposition}

\theoremstyle{definition}
\newtheorem{definition}[theorem]{Definition}
\newtheorem{example}[theorem]{Example}

\theoremstyle{remark}
\newtheorem{remark}[theorem]{Remark}

\numberwithin{equation}{subsection}

\DeclareMathOperator{\id}{id}

\DeclareMathOperator{\kk}{\Bbbk}
\DeclareMathOperator{\tv}{\tilde{\varphi}}
\DeclareMathOperator{\tp}{\tilde{\partial}}

\DeclareMathOperator{\FF}{\mathbb{F}}

\DeclareMathOperator{\ch}{char}

\DeclareMathOperator{\End}{End}

\title{Keller maps and holonomic $D$-modules}
\author{Wenliang Zhang}

\address{Department of Mathematics, Statistics, and Computer Science, University of Illinois at Chicago, 851 S. Morgan Street, Chicago, IL 60607-7045}
\email{wlzhang@uic.edu}

\begin{document}
\maketitle

\begin{abstract}
Let $\kk$ be a field of characteristic $p$ and $R=\kk[x_1,\dots,x_n]$. Let $D$ denote the ring of $\kk$-linear differential operators on $R$. Let $\varphi:R\to R$ be a $\kk$-algebra endomorphism whose Jacobian is a unit in $R$. We show that there is a unique $\kk$-algebra endomorphism $\Phi:D\to D$ such that $\Phi|_R=\varphi$ and that the restriction of scalar via $\Phi$ preserves holonomicity. We also construct an example of a $\kk$-algebra endomorphism of $D$ that does not preserve holonomicity.
\end{abstract}

\section{Introduction}
Let $\kk$ be a field of characteristic 0 and $R=\kk[x_1,\dots,x_n]$ be a polynomial ring in $n$ variables. Let $\varphi:R\to R$ be a $\kk$-algebra endomorphism. $\varphi$ is called {\it Keller} (or \emph{Jacobian} by some authors) if $\det(J_{\varphi})\in \kk^\times$ where $J_{\varphi}$ denotes the Jacobian matrix of $\varphi$. The Jacobian conjecture (\cite{Keller1939}) asserts that every Keller endomorphism on $R$ is an automorphism. This conjecture has been investigated from a viewpoint of $D$-modules ({\it cf.} \cite{Dixmier1968, BCW82, Bass89, BavulaDixmierConjecture, BKKpaper}), where $D$ stands for the ring of $\kk$-linear differential operators on $R$ (more details on $D$ can be found in \S\ref{sec: Keller}). Let $\tilde{\varphi}$ denote a $\kk$-algebra endomorphism on $D$ such that $\tv|_R=\varphi$ (the existence of $\tv$ is discussed in detail in \S\ref{sec: Keller}). Let $\prescript{\tv}{}R$ denote the $D$-module that agrees with $R$ as an abelian group with the (left) $D$-structure given by $\theta\cdot r:=\tv(\theta)r$ for each $\theta\in D$. If $\prescript{\tilde{\varphi}}{}R$ remains a simple $D$-module, then the Jacobian conjecture holds true (Remark \ref{connection between extension and jacobian}). The recent counterexample to the Jacobian conjecture (\cite{Alpoge2026}) implies that $\prescript{\tv}{}R$ may not be a simple $D$-module. From the $D$-module theoretic viewpoint, $\prescript{\tv}{}R$ being a finite-length $D$-module is the best one can hope for. Indeed, a remarkable result of Bavula (\cite[Theorem 1.3]{BavulaAnnals}) shows that if $M$ is a holonomic $D$-module then so is $\prescript{\tv}{}M$; in particular $\prescript{\tv}{}R$ has finite length as a $D$-module. (The notion of holonomic $D$-modules is discussed in detail in \S\ref{sec: holonomic}.)

Recently there have been considerable developments and interest in $D$-modules in prime characteristic $p$, including the notion of holonomic $D$-modules. It is natural to ask whether \cite[Theorem 1.3]{BavulaAnnals} can be extended to prime characteristic $p$, which is the motivation behind this article. The ring of differential operators behaves vastly differently in characteristic $p$ from characteristic 0: in characteristic 0, $D$ is has finitely many $\kk$-algebra generators $\{x_1,\dots,x_n,\partial_{x_1},\dots,\partial_{x_n}\}$ and consequently an extension of a Keller map in $\End_{\kk}(D)$ is determined by its image on these finitely many generators.However, in characteristic $p$, the ring $D$ is no longer finitely generated as a $\kk$-algebra. Set $\partial_i^{[t]}:=\frac{1}{t!}\frac{\partial^{t}}{\partial x^t_i}$. Then $\{\partial_i^{[p^e]}\mid e\geq 1\}$ is an algebraically independent set over $\kk$. Thus an extension of a Keller map in $\End_{\kk}(D)$ can no longer be determined by finitely many elements. Consequently new approaches are necessary for extending \cite[Theorem 1.3]{BavulaAnnals} to prime characteristic $p$.

Our first main result, proved in \S\ref{sec: Keller}, is the following:
\begin{theorem}[=Theorem \ref{keller admit unique ext in char p}]
\label{main thm 1}
Assume $\ch(\kk)=p>0$ and $R=\kk[x_1,\dots,x_n]$. Let $\varphi\in \End_{\kk}(R)$ be a Keller map. Then there is a unique $\Phi\in \End_{\kk}(D)$ such that $\Phi|_R=\varphi$.
\end{theorem}

This is surprising and drastically different from characteristic 0. In characteristic 0, a Keller map may admit infinitely many extensions in $\End_{\kk}(D)$ (Example \ref{infinitely many ext of Keller}). Moreover, in characteristic $p$, if $\varphi$ is not Keller ({\it e.g.} the Frobenius endomorphism on $R$), then it may also admit infinitely many extensions in $\End_{\kk}(D)$ ({\it cf.} Example \ref{ex: infinite extension of Frob}).

To finish our extension of \cite[Theorem 1.3]{BavulaAnnals}, we prove in \S\ref{sec: holonomic} that:
\begin{theorem}[=Theorem \ref{thm: psi preserve holonomic}]
\label{main thm 2}
Assume $\ch(\kk)=p>0$ and $R=\kk[x_1,\dots,x_n]$. Let $\varphi\in \End_{\kk}(R)$ be a Keller map and let $\Phi$ be its unique extension in $\End_{\kk}(D)$. Then the restriction of scalar via $\Phi$ preserves holonomicity; that is, if $M$ is a holonomic $D$-module, then so is $\prescript{\Phi}{}M$.
\end{theorem}

If $\varphi\in \End_{\kk}(R)$ is not Keller ({\it e.g.} the Frobenius endomorphism), then $\varphi$ may admit an extension in $\End_{\kk}(D)$ that does not preserve holonomicity. In \S\ref{counter example}, we construct such an example (Example \ref{example not preserve holo}). Hence in characteristic $p$:
\begin{enumerate}
\item if $\varphi\in \End_{\kk}(R)$ is Keller, then it admits a unique extension $\Phi\in \End_{\kk}(D)$ and the restriction of scalar via $\Phi$ preserves holonomicity, or
\item if $\varphi\in \End_{\kk}(R)$ is {\it not} Keller, then it may admit an extension $\tv\in \End_{\kk}(D)$ such that the restriction of scalar via $\tv$ does {\it not} preserve holonomicity.
\end{enumerate}

Nonetheless, combining Theorem \ref{main thm 1}, Theorem \ref{main thm 2} and \cite[Theorem 1.3]{BavulaAnnals} produces the following characteristic-free statement:
\begin{theorem}
Let $\kk$ be an arbitrary field and $R=\kk[x_1,\dots,x_n]$. Assume that $\varphi\in \End_{\kk}(R)$ is Keller and let $\tv$ be any extension of $\varphi$ in $\End_{\kk}(D)$. Then the restriction of scalar via $\tv$ preserves holonomicity; that is, if $M$ is a holonomic $D$-module, then so is $\prescript{\tv}{}M$. In particular, $\prescript{\tv}{}R$ has finite length in the category of $D$-modules.
\end{theorem}

\section{Keller maps and their extensions}
\label{sec: Keller}
Let $R=\kk[x_1,\dots,x_n]$ where $\kk$ is a field. We will denote the ring of $\kk$-algebra endomorphisms on $R$ (respectively $D$) by $\End_{\kk}(R)$ (respectively $\End_{\kk}(D)$). Let $\varphi\in \End(R)$ be a $\kk$-algebra endomorphism of $R$. Set $\varphi_i=\varphi(x_i)$. Let $J_{\varphi}$ denote the Jacobian matrix $\begin{pmatrix} \frac{\partial \varphi_i}{\partial x_j} \end{pmatrix}$. A $\kk$-algebra endomorphism $\varphi:R\to R$ is called \emph{Keller} if $\det(J_\varphi)\in k^{\times}$.

We will denote by $D_{R/\kk}$ (or $D$ in short whenever $R$ and $\kk$ are clear from the context) the ring of $\kk$-linear differential operators on $R$ (in the sense of \cite[\S16]{EGA4.4}). More specifically, since $R$ is a polynomial ring over $\kk$, the ring $D$ can be described explicitly as follows:
\[D=R\langle \partial^{[t]}_i\mid t\geq 1,\ \ i=1,\dots,n\rangle\]
where $\partial^{[t]}_i:=\frac{1}{t!}\frac{\partial^t}{\partial x^t_i}:R\to R$ is the $\kk[x_1,\dots,x_{i-1},x_{i+1},\dots,x_n]$-linear map that sends $x^s_i$ to $\binom{s}{t}x^{s-t}_i$. 

It is clear that every Keller map on $R$ is \'{e}tale. We will recall  properties of \'{e}tale maps related to differential operators as follows:
\begin{remark}
\label{Masson theorem}
Let $f:A\to B$ be an \'{e}tale map of $\kk$-algebras of finite type. Let $D_{A/\kk}$ (respectively $D_{B/\kk}$) denote the ring of $\kk$-linear differential operators on $A$ (respectively on $B$). Then it follows from \cite[Theorem 2.2.5]{MassonThesis} that $f$ induces a unique $\kk$-algebra homomorphism $\tilde{f}:D_{A/\kk}\to D_{B/\kk}$ such that $\tilde{f}(\theta)$ is the unique element in $D_{B/\kk}$ that satisfies 
\[\tilde{f}(\theta)(f(a))=f(\theta(a))\quad\forall a\in A.\]
Put it differently, $f$ induces a unique $\kk$-algebra homomorphism $\tilde{f}:D_{A/\kk}\to D_{B/\kk}$ such that the following diagram commutes for every $\theta\in D_{A/\kk}$
\[
\xymatrix{
A \ar[r]^{\theta} \ar[d]^{f} & A \ar[d]^{f}\\
B\ar[r]^{\tilde{f}(\theta)} & B
}
\]
  
Moreover, the order of differential operators does not increase under this specific $\tilde{f}$; that is, if the order of $\theta$ is $\leq t$ then so is the order of $\tilde{f}(\theta)$.

Furthermore, if $f:A\to B$ is an \'{e}tale map of $\kk$-algebras of finite type, according to \cite[Theorem 2.2.10]{MassonThesis}, $f$ induces an isomorphism
\[B\otimes_AD_{A/\kk}\xrightarrow[\sim]{b\otimes \delta\mapsto b\tilde{f}(\delta)}D_{B/\kk}.\]
\end{remark}

Remark \ref{Masson theorem} prompts the following definition.
\begin{definition}
Let $R=\kk[x_1,\dots,x_n]$ where $\kk$ is a field and $\varphi:R\to R$ be an endomorphism. Set $D=D_{R/\kk}$.
\begin{enumerate}

\item $\tilde{\varphi}\in \End_{\kk}(D)$ is called an \emph{extension} of $\varphi$ if $\tilde{\varphi}(r)=\varphi(r)$ for every $r\in R$.

\item An extension $\tilde{\varphi}\in \End_{\kk}(D)$ of $\varphi$ is called \emph{functorial} if the following diagram commutes for every $\theta\in D$:
\[
\xymatrix{
R \ar[r]^{\theta} \ar[d]^{\varphi} & R \ar[d]^{\varphi}\\
R\ar[r]^{\tilde{\varphi}(\theta)} & R
}
\]
that is, 
\[\tilde{\varphi}(\theta)(\varphi(r))=\varphi(\theta(r))\quad\forall r\in R.\]
\item An extension $\tilde{\varphi}:D\in \End_{\kk}(D)$ of $\varphi\in \End_{\kk}(R)$ is called \emph{strict} if it does not increase the order of differential operators; that is, if $\theta\in D$ has order $\leq t$, then $\tilde{\varphi}(\theta)$ has order $\leq t$ for every $\theta\in D$.
\end{enumerate}
\end{definition}

Since Keller maps are \'{e}tale, the following is an immediate consequence of Remark \ref{Masson theorem}:
\begin{proposition}
\label{Keller admit functorial}
Let $R=\kk[x_1,\dots,x_n]$ where $\kk$ is an arbitrary field. Let $\varphi:R\to R$ be a Keller map.
\begin{enumerate}
\item $\varphi$ admits a unique functorial extension $\Phi$ in $\End_{\kk}(D)$.
\item This unique functorial extension $\Phi$ is strict.
\item 
\label{Keller induce basis of D}
$\{\Phi(\partial^{[b_1]}_1)\cdots \Phi(\partial^{[b_n]}_n))\mid b_1,\dots,b_n\geq 0\}$ is an $R$-basis of $D$; that is, 
\[D= \bigoplus_{b_1,\dots,b_n\geq 0} R \Phi(\partial^{[b_1]}_1)\cdots \Phi(\partial^{[b_n]}_n).\]
\end{enumerate}
\end{proposition}

\begin{remark}
\label{connection between extension and jacobian}
Assume $\ch(\kk)=0$. The functorial extensions of Keller maps are closely related to the Jacobian conjecture. Let $\varphi:R\to R$ be a Keller map. Let $\tv$ be the functorial extension of $\varphi$ to $\End_{\kk}(D)$. We claim that, if $\prescript{\tv}{}R$ is a simple $D$-module, then $\varphi$ must be an automorphism (hence the Jacobian conjecture holds true). To see this, we consider the image of $\varphi$. Since $\tv(\theta)(r)=\varphi(\theta(r))$ by functoriality, the image $\varphi(R)$ is naturally a $D$-submodule of $\prescript{\tv}{}R$. If $\prescript{\tv}{}R$ is a simple $D$-module, then $\varphi(R)=R$. Since $R$ is noetherian, surjectivity of $\varphi$ forces injectivity; this shows that $\varphi$ is an automorphism.

However, the recent counterexample to the Jacobian conjecture implies that $\prescript{\tv}{}R$ may not be a simple $D$-module. Hence, in terms of $D$-modules, the finite length property of $\prescript{\tv}{}R$ is the best one can hope for. 
\end{remark}

Functorial extensions may not be strict as shown in the following example.
\begin{example}
Let $R=\FF_p[x]$. Let $F:R\to R$ denote the Frobenius endomorphism. Define $\tilde{F}:D\to D$ via $\tilde{F}(r)=F(r)=r^p$ for $r\in R$ and $\tilde{F}(\partial^{[t]})=\partial^{[pt]}$ for $t\geq 1$. Then one can check that $\tilde{F}$ is a functorial extension, but not a strict extension as it increases the order of differential operators.
\end{example}

Strict extensions may not be functorial as shown in the following example.
\begin{example}
\label{infinitely many ext of Keller}
Assume $\ch(\kk)=0$ and let $R=\kk[x]$. Consider identity map $\id_R$ on $R$. For each $f\in R$, define $\phi_f:D\to D$ via $\phi_f(r)=r$ and $\phi_f(\partial)=\partial+f(x)$. Then one can check that:
\begin{enumerate}
\item Each $\phi_f$ is a strict extension of $\id_R$ in $\End_{\kk}(D)$ (hence $\id_R$ admits infinitely many extensions in $\End_{\kk}(D)$), and that
\item $\phi_f$ is functorial if and only if $f=0$. The unique functorial extension of $\id_R$ in $\End_{\kk}(D)$ is indeed $\id_D$. 
\end{enumerate}
\end{example}

\begin{remark}[Nousiainen's Theorem]
Assume now that $\ch(k)=p>0$. Let $\varphi$ be a Keller map. Then it follows from Nousiainen's Theorem (\cite{NousiainenThesis}) that $\{\varphi_1\dots,\varphi_n\}$ is a $p$-basis for $R$ over $R_{(1)}:=k[R^p]=k[x^p_1,\dots,x^p_n]$; that is, 
\[R= \bigoplus_{0\leq a_1,\dots,a_n\leq p-1}R_{(1)}\varphi^{a_1}_a\cdots \varphi^{a_n}_n.\]
If we set $R_{(e)}:=k[R^{p^e}]=k[x^{p^e}_1,\dots,x^{p^e}_n]$ for each $e\geq 1$, then 
\[R_{(e)}= \bigoplus_{0\leq a_1,\dots,a_n\leq p-1}R_{(e+1)}\varphi^{a_1p^e}_a\cdots \varphi^{a_np^e}_n\ {\rm and}\ R= \bigoplus_{0\leq a_1,\dots,a_n\leq p^e-1}R_{(e)}\varphi^{a_1}_a\cdots \varphi^{a_n}_n.\]
\end{remark}

Let $\varphi$ be a $\kk$-algebra endomorphism on $R$ that admits an extension $\tilde{\varphi}\in \End_{\kk}(D)$. Since $\tilde{\varphi}$ is a ring map, the following hold:
\begin{equation}
\begin{aligned}
\label{action of d on phi}
[\tilde{\varphi_i}(\partial_i), \varphi_j] & =\tilde{\varphi}([x_i,\partial_j])=\tilde{\varphi}(\delta_{ij})=\delta_{ij} \\
[\tilde{\varphi_i}(\partial^{[p^e]}_i), \varphi^{p^e}_j] & =\delta_{ij} \\
\tilde{\varphi}(\partial^{[t]}_i)(\varphi^{m}_i) & =\begin{cases}\binom{m}{t}\varphi^{m-t}_i & m\geq t \\0 & m<t\end{cases}  \\
[\tilde{\varphi}(\partial^{[t]}_i), \tilde{\varphi}(\partial^{[t]}_j)] & =0 
\end{aligned}
\end{equation}

\begin{lemma}
\label{keller acts as variable}
Let $R=\kk[x_1,\dots,x_n]$ where $\kk$ is an arbitrary field. Let $\varphi\in \End_{\kk}(R)$ be a Keller map and set $\varphi_i:=\varphi(x_i)$. Then
\[R=\{\theta\in D| [\theta, \varphi_i]=0,\ {\rm for}\ i=1,\dots,n\}.\]
\end{lemma}
\begin{proof}
Let $\Phi$ denote the unique functorial extension of $\varphi$ (Proposition \ref{Keller admit functorial}). By Proposition \ref{Keller admit functorial}(\ref{Keller induce basis of D}), $\theta$ can be written as 
\[\theta=\sum_{b_1,\dots,b_n\geq 0} r_{b_1,\dots,b_n}\Phi(\partial^{[b_1]}_1)\cdots \Phi(\partial^{[b_n]}_n).\]
The commutator $[\theta,\varphi_1]=0$ together with $[\Phi(\partial_i),\varphi_j]=\delta_{ij}$ implies that $r_{b_1,\dots,b_n}=0$ whenever $b_1\geq 1$. Continuing this process, one can see that the only potentially nonzero coefficient is $r_{0,\dots,0}$ and hence $\theta\in R$. This shows one inclusion $\supseteq$. The other inclusion $\subseteq$ is clear.
\end{proof}

\begin{remark}
\label{diff oper in char p}
Assume that $\ch(\kk)=p>0$ and $R=\kk[x_1,\dots,x_n]$. Then, for each $i$, the divided power differential operators $\{\partial^{[p^e]}_i\mid e\geq 0\}$ is algebraically independent. On the other hand, if $m=\sum_{i=0}^t m_ip^i$ with $0\leq m_i\leq p-1$, then. one can check
\[\partial^{[m]}_i=\prod_{i=0}^t (\partial^{[p^i]})^{m_i}.\]
Consequently, if $m$ is not a power of $p$, then $\partial^{[m]}_i$ is determined by differential operators of lower orders.
\end{remark}

The following lemma is stated in the form that is most convenient to this article.

\begin{lemma}
\label{decomposition thm}
Let $A=k[y_1,\dots,y_n]$ and let M be an $A$-module. Set $A_{(1)}:=\kk[y^p_1,\dots,y^p_n]$. Assume that there is an $R_{(1)}$-module morphism $\phi:M\to M$ and an element $a\in A$ such that $\phi^p=0$ and $\phi\cdot a-a\cdot \phi=\id_M$, then 
\[M = \oplus_{i=0}^{p-1}a^i\ker(\phi)\]
in the category of $A_{(1)}$-modules.
\end{lemma}
\begin{proof}
Set $M_j:=\ker(\phi^{j+1})$ for $j\geq 0$. Note that $M=M_{p-1}$ by the hypothesis. First we show that $M=\sum_{i=0}^{p-1}a^i(\ker(\phi))$. Let $z$ be an element of $M$ and we wish to show that $z$ can be written as $z=\sum_{i=0}^{p-1}a^iz_i$ for some $z_i\in \ker(\phi)$. Since $M=M_{p-1}$, there is $0\leq j\leq p-1$ such that $z\in M_j$. We will induct on $j$. If $j=0$ (that is, $z\in M_0=\ker(\phi)$), then we can set $z_0=z$ and $z_i=0$ for $i\geq 1$. Assume that $j\geq 1$ and that we have proved the decomposition for nonnegative integers $<j$. Set $z':=\frac{1}{j!}\phi^j(z)$ (since $j<p$, $j!\neq 0$). It follows that $z'\in \ker(\phi)$ since $\phi(z')=\phi^{j+1}(z)=0$. And, applying $\phi\cdot a-a\cdot \phi=\id_M$ and $\phi(z')=0$, one can check:
\[\phi^j(a^jz')=j!z'=\phi^j(z).\]
Consequently $\phi^j(z-a^jz')=\phi^j(z)-\phi^j(a^jz')=0$ which implies that $z-a^jz'\in M_{j-1}$. By induction, $z-a^jz'=\sum_{i=0}^{p-1}a^iz_i$ with $z_i\in \ker(\phi)$ which provides the desired decomposition of $z$. This finishes the proof of $M=\sum_{i=0}^{p-1}a^i\ker(\phi)$.

To finish the proof of the desired direct sum decomposition of $M$, we will prove that if $\sum_{i=0}^{p-1}a^iz_i=0$ with $z_i\in \ker(\phi)$ then $z_0=\cdots=z_{p-1}=0$. Note that (by applying $\phi\cdot a-a\cdot \phi=\id_M$ and $\phi(z_i)=0$)
\[\phi^{p-1}(a^iz_i)=\begin{cases} 0 & i<p-1\\ (p-1)!z_{p-1} & i=p-1\end{cases}.\] 
Hence applying $\phi^{p-1}$ to $\sum_{i=0}^{p-1}a^iz_i=0$ produces $(p-1)!z_{p-1}=0$ and hence $z_{p-1}=0$. One can then follows this process to conclude $z_0=\cdots=z_{p-1}=0$. This finishes the proof of our theorem.
\end{proof}

We are now in position to prove our first main result. As a reminder, one key feature of $\partial^{[t]}_i$ in characteristic $p$ is that $(\partial^{[t]}_i)^p=0$ for each $i=1,\dots,n$ and each $t\geq 1$.

\begin{theorem}
\label{keller admit unique ext in char p}
Let $R=\kk[x_1,\dots,x_n]$. Assume $\ch(\kk)=p>0$ and $\varphi:R\to R$ is a Keller map. Then $\varphi$ admits a unique extension in $\End_{\kk}(D)$. In particular, the functorial extension $\Phi$ is the unique extension of $\varphi$ in $\End_{\kk}(D)$.
\end{theorem}
\begin{proof}
Let $\tilde{\varphi}:D\to D$ be an endomorphism such that $\tilde{\varphi}|_R=\varphi$. It suffices to show that
\[\tilde{\varphi}(\partial^{[p^e]}_i)=\Phi(\partial_i^{[p^e]}),\ \forall e\geq 0,\ i=1,\dots,n.\]
To this end, we will induct on $e$.

We start with the case when $e=0$; that is, we wish to show that $\tilde{\varphi}(\partial_i)=\Phi(\partial_i)$ for $i=1,\dots,n$. It follows from (\ref{action of d on phi}) that $[(\Phi-\tilde{\varphi})(\partial_i),\varphi_j]=0$ for $i,j=1,\dots,n$. It follows from Lemma \ref{keller acts as variable} that $(\Phi-\tilde{\varphi})(\partial_i)\in R$. Set $r_i=(\Phi-\tilde{\varphi})(\partial_i)\in R$. That is $\tv(\partial_i)=\Phi(\partial_i)+r_i$. In this notation, $r_i$ stands for multiplication by $r_i$ on R. Since the order of $\Phi(\partial_i)$ is at most the order of $\partial_i$, it follows that $\Phi(\partial_i)$ commutes with elements in $R_{(1)}$. Hence $\Phi(\partial_i)+r_i$ commutes with elements in $R_{(1)}$, and so does $\tv(\partial)$. Hence $\tv(\partial_i)\in \End_{R_{(1)}}(R)$ for each $i=1,\dots,n$. In particular, $\ker(\tv(\partial_1))$ is an $R_{(1)}$-submodule of $R$. Since $\partial^p_1=0$, we have $\tv(\partial_1)^p=0$. Since $x\partial_1x_1-x_1\partial_1=1$, we have $\tv(\partial_1)\varphi_1-\varphi_1\tv(\partial_1)=1$. That is, $\tv(\partial_1)$ and $\varphi_1$ satisfy the assumptions in Lemma \ref{decomposition thm}. It now follows from Lemma \ref{decomposition thm} that 
\[R= \bigoplus_{0\leq a_1\leq p-1}\varphi_1^{a_1}\ker(\tv(\partial_1)).\]
Since $\tv(\partial_i)$ commutes with $\tv(\partial_j)$ for all $i,j$, it follows that $\tv(\partial_2)\in \End_{R_1}(\ker(\tv(\partial_1)))$. Applying Lemma \ref{decomposition thm} to the case when $M=\ker(\tv(\partial_1))$ and $\phi=\tv(\partial_2)$ produces
\[R= \bigoplus_{0\leq a_1,a_2\leq p-1}\varphi_1^{a_1}\varphi^{a_2}_2\left(\ker(\tv(\partial_1))\cap \ker(\tv(\partial_2))\right).\]
Continuing this process, we have
\[R= \bigoplus_{0\leq a_1,\dots,a_n\leq p-1}\varphi_1^{a_1}\cdots \varphi^{a_n}_n\left(\ker(\tv(\partial_1))\cap\cdots \cap \ker(\tv(\partial_n))\right).\]
Set $K:=\ker(\tv(\partial_1))\cap\cdots \cap \ker(\tv(\partial_n))$. Then $K$ is a direct summand of $R$ as an $R_{(1)}$-module. Hence $K$ is a projective $R_{(1)}$-module. By counting the number of summands, $K$ has rank 1 over $R_{(1)}$. Since $R_{(1)}$ is a polynomial ring in the variables $x^p_1,\dots,x^p_n$, it follows that $K\cong R_{(1)}$. Let $z$ be a generator of $K$, then $R_{(1)}\cong K=R_{(1)}z$. Consequently

\begin{align*}
R&= \bigoplus_{0\leq a_1,\dots,a_n\leq p-1}\varphi_1^{a_1}\cdots \varphi^{a_n}_nK\\
&=\bigoplus_{0\leq a_1,\dots,a_n\leq p-1}\varphi_1^{a_1}\cdots \varphi^{a_n}_nR_{(1)}z\\
&=(\bigoplus_{0\leq a_1,\dots,a_n\leq p-1}\varphi_1^{a_1}\cdots \varphi^{a_n}_nR_{(1)})z\\
&= Rz
\end{align*}

where the last isomorphism holds since $R= \bigoplus_{0\leq a_1,\dots,a_n\leq p-1}\varphi_1^{a_1}\cdots \varphi^{a_n}_nR_{(1)}$. That is, we have proved $R=Rz$. It follows that $z\in R^\times=k^\times$. Since both $\tv(\partial_i)$ and $\Phi(\partial_i)$ are $k$-linear differential operators,
\[0=\tv(\partial_i)(z)=(\Phi(\partial_i)+r_i)(z)=0+r_iz\]
which implies that $r_i=0$ for each $i=1,\dots,n$. This shows that $\tv(\partial_i)=\Phi(\partial_i)$ for each $i=1,\dots,n$. This completes the proof of the base case when $e=0$ for our induction.

Assume that we have proved $\tv(\partial^{[p^0]}_i)=\Phi(\partial^{[p^0]}_i),\dots, \tv(\partial^{[p^e]}_i)=\Phi(\partial^{[p^e]}_i)$ for each $i=1,\dots,n$. We wish to show that $ \tv(\partial^{[p^{e+1}]}_i)=\Phi(\partial^{[p^{e+1}]}_i)$ for each $i=1,\dots,n$.

One immediate consequence of our hypothesis $\tv(\partial^{[p^0]}_i)=\Phi(\partial^{[p^0]}_i),\dots, \tv(\partial^{[p^e]}_i)=\Phi(\partial^{[p^e]}_i)$ and Remark \ref{diff oper in char p} is that
\[\tv(\partial^{[t]}_i)=\Phi(\partial^{[t]}_i)\quad {\rm for}\ 0\leq t\leq p^{e+1}-1.\]
Consider $\Phi(\partial^{[p^{e+1}]}_i)-\tv(\partial^{[p^{e+1}]}_i)$. Since $[\partial^{[p^{e+1}]}_i, x_j]=0$ for $i\neq j$ and $[\partial^{[p^{e+1}]}_i, x_i]=\partial^{[p^{e+1}-1]}_i$, we have $[\tv(\partial^{[p^{e+1}]}_i), \varphi_j]=0=[\partial^{[p^{e+1}]}_i, x_j]$ for $i\neq j$ and 
\[[\tv(\partial^{[p^{e+1}]}_i), \varphi_i]=\tv(\partial^{[p^{e+1}-1]}_i)=\Phi(\partial^{[p^{e+1}-1]}_i)=[\Phi(\partial^{[p^{e+1}]}_i), \varphi_i].\]
Consequently $[(\Phi(\partial^{[p^{e+1}]}_i)-\tv(\partial^{[p^{e+1}]}_i)), \varphi_j]=0$ for all $i,j=1,\dots,n$. Repeating the previous argument, we have $\tv(\partial^{[p^{e+1}]}_i)-\Phi(\partial^{[p^{e+1}]}_i)=a_i\in R$ for each $i=1,\dots,n$. Also, by the previous argument in the case when the underlying ring is $R_{(e+1)}$, we have that 
\[R_{(e+1)}= \bigoplus_{0\leq a_1,\dots,a_n\leq p-1}\varphi_1^{a_1p^{e+1}}\cdots \varphi^{a_np^{e+1}}_n\left(\ker(\tv(\partial^{[p^{e+1}]}_1))\cap\cdots \cap \ker(\tv(\partial^{[p^{e+1}]}_n))\right).\]
Set $K:=\ker(\tv(\partial^{[p^{e+1}]}_1))\cap\cdots \cap \ker(\tv(\partial^{[p^{e+1}]}_n))$. Then it follows that $K$ is a rank-1 projective $R_{(e+2)}$-module and hence $R_{(e+2)}\cong K=R_{(e+2)}g$ for a generator of $g$ of $K$. Then one can follow the same argument as before to show that $g\in k^\times$ and that $a_i=0$ for each $i=1,\dots,n$. This shows that $\Phi(\partial^{[p^{e+1}]}_i)=\tv(\partial^{[p^{e+1}]}_i)$ for each $i=1,\dots,n$ and completes the proof of our theorem.
\end{proof}

If $\varphi$ is not Keller, then its extensions in $\End_{\kk}(D)$ may no longer be unique. This phenomenon is already noted in \cite{BavulaExtensionFrobenius} for the Frobenius endomorphism\footnote{We thank Vladimir V. Bavula for pointing this out to us.}.

We consider the following specific example due to its relevance to Example \ref{example not preserve holo}.
\begin{example}
\label{ex: infinite extension of Frob}

Consider the case when $R=\FF_p[x]$ and $\varphi=F$ is the Frobenius endomorphism. It admits a natural extension $\tilde{F}$ defined via
\[\tilde{F}(x)=x^p\quad \tilde{F}(\partial^{[t]})=\partial^{[pt]}\]
for all $t\geq 1$. We will modify $\tilde{F}$ to produce infinitely many different extensions of $F$ in $\End_{\FF_p}(D)$. 
Consider the projection 
\[\pi_0: R=\oplus_{j=0}^{p-1}x^j\FF_p[x^p]\to \FF_p[x^p]\] 
such that $\pi_0(x^m)=\begin{cases} x^m& p|m\\ 0 & {\rm otherwise}\end{cases}$. It turns out that $\pi_0$ agrees with $\sum_{i=0}^{p-1}(-1)^ix^i\partial^{[i]}\in D$. Note that $\pi_0x\pi_0(r)=0$ for every $r\in R$ and hence $\pi_0x\pi_0=0\in D$. Set $\pi=x\pi_0$. Then $\pi^2=0\in D$.

For each $e\geq 1$, set 
\[a_e:=p^{e+1}+p^{4e+2}\quad {\rm and} \quad w_e:=x^{a_e}\pi.\] 
Then one can check that $w^2_e=0$ in $D$ since $x^{a_e}\pi=\pi x^{a_e}$ and $\pi^2=0$. Now define $\alpha_e:D\to D$ by 
\[\alpha_e(\theta)=(1+w_e)\theta(1-w_e)\quad \forall\ \theta\in D.\] 
One can check $\alpha_e$ is a $\FF_p$-algebra endomorphism on $D$ since $(1+w_e)(1-w_e)=1$. Finally we define
\[\tilde{F}_e:=\alpha_e\circ \tilde{F}\]
Then one can check that $\tilde{F}_e(r)=F(r)=r^p$ for each $r\in R$ and that 
\begin{equation}
\label{alpha on small t}
\tilde{F}_e(\partial^{[t]})=\alpha_e(\partial^{[pt]})=\begin{cases}\partial^{[pt]}& t<p^e\\ \partial^{[p^{e+1}]} - x^{p^{4e+2}}\pi& t=p^e\end{cases}
\end{equation}
which shows that $\tilde{F}_e$ are different for different $e$.
\end{example}

\section{Holonomic $D$-modules}
\label{sec: holonomic}
\subsection{Holonomic $D$-modules.} We follow closely the characteristic-free approach to holonomic $D$-modules developed in \cite{LyubeznikHolonomic} which was inspired by the notion of holonomicity introduced in \cite{BavulaHolonomic}. 

Let $\kk$ be a field and $R=\kk[x_1,\dots,x_n]$. Let $D$ denote the ring of $\kk$-linear differential operators on $R$. Recall that: 
\[D=R\langle \partial^{[t]}_i\mid t\geq 1\ i=1,\dots,n\rangle\]
where $\partial^{[t]}_i:=\frac{1}{t!}\frac{\partial^t}{\partial x^t_i}:R\to R$ is the $\kk[x_1,\dots,x_{i-1},x_{i+1},\dots,x_n]$-linear map that sends $x^s_i$ to $\binom{s}{t}x^{s-t}_i$. For each integer $i\geq 0$, set $\scr{F}_i$ to be the $\kk$-span of the set $\{x^{a_1}_1\cdots x^{a_n}_n\partial^{[b_1]}_1\cdots \partial^{[b_n]}_n\mid a_1+\cdots+a_n+b_1+\cdots+b_n\leq i\}$ with the convention $\scr{F}_0=\kk$. Then $\{\scr{F}_i\}$ is an ascending chain of finite dimensional $\kk$-subspaces of $D$ such that $\bigcup_i\scr{F}_i=D$ and that $\scr{F}_i\scr{F}_j\subseteq \scr{F}_{i+j}$. This filtration $\{\scr{F}_i\}$ is called the \emph{canonical filtration}\footnote{It is called the {\it Bernstein filtration} in \cite{LyubeznikHolonomic}} of $D$ in \cite{BavulaHolonomic}.

In this article, we will focus on left $D$-modules. 

\begin{definition}
\label{defn: holonomic}
Let $M$ be a $D$-module. A \emph{$\kk$-filtration} on $M$ is an ascending chain of finite dimensional $\kk$-subspaces of $M$: $M_0\subseteq M_1\subseteq \cdots$ such that $\bigcup_iM_i=M$ and $\scr{F}_iM_j\subseteq M_{i+j}$.

A $D$-module $M$ is called \emph{holonomic} if it admits a $\kk$-filtration $\{M_i\}$ and a constant $c$ such that $\dim_{\kk}(M_i)\leq ci^n$ for all $i\geq 0$ (as a reminder, $n=\dim(R)$ and $c$ is independent of $i$).
\end{definition}

\begin{remark}
The notion of holonomic $D$-modules in \cite[p.198]{BavulaHolonomic} is introduced for finitely generated $D$-modules in prime characteristic $p$. In this article, we follow the characteristic-free approach in \cite[Definition 3.4]{LyubeznikHolonomic} which does assume $M$ to be finitely generated. Even though $M$ is not assumed to be a finitely generated $D$-module in Definition \ref{defn: holonomic} (or \cite[Definition 3.4]{LyubeznikHolonomic}), it is proved in \cite[Theorem 3.5]{LyubeznikHolonomic} that a holonomic $D$-module have finite length in the category of $D$-modules ({\it cf.} \cite[Theorem 9.6]{BavulaHolonomic}), a fortiori, finitely generated.
\end{remark}

\begin{remark}
\label{holonomic is cyclic}
If a $D$-module $M$ has finite length, then it is cyclic. In particular, every holonomic $D$-module is cyclic. 

This is well-known in characteristic 0 and proved in \cite[Theorem 9.7]{BavulaHolonomic} in prime characteristic $p$. 
\end{remark}

The next result (Lemma \ref{filtration by generator}) is developed in \cite{LZ26}, a collaboration with Gennady Lyubeznik. We thank him for allowing it to be included here.

\begin{lemma}
\label{filtration by generator}
Let $M$ be a cyclic $D$-module with a generator $z$. Set $\tilde{M}_i:=\scr{F}_i\cdot z$ for each $i\geq 0$. If $M$ is holonomic, then there is a constant $C$ such that $\dim_{\kk}(\tilde{M}_i)\leq Ci^n$ for all $i\leq 0$.
\end{lemma}
\begin{proof}  
As $M$ is holonomic, it admits a $\kk$-filtration $M_0\subset M_1\subset\dots$ such that $\dim_{\kk}M_i\leq ci^n$ for a constant c and for all $i$. Let $t$ be an integer such that $z\in M_t$ (such an $t$ exists as $\{M_i\}$ is exhaustive). Then $\tilde{M}_i\subset M_{i+t}$ for all $i$, hence $\dim_{\kk}\tilde{M}_i\leq {\rm dim}_kM_{i+t}\leq c(i+t)^n\leq c(1+t)^ni^n$ for all $i\geq 0$.
\end{proof}

\begin{definition}
Let $M$ be a $D$-module and let $\phi\in \End_{\kk}(D)$ be a $\kk$-algebra endomorphism of $D$. Then we define $\prescript{\phi}{}M$ to be the $D$-module whose underlying abelian is the same as $M$ with $D$-structure defined via:
\[\theta\cdot m:=\phi(\theta)\cdot m\]
for all $\theta\in D$ and $m\in M$.
\end{definition}

The following is an immediate consequence of Remark \ref{holonomic is cyclic} and Lemma \ref{filtration by generator}.
\begin{proposition}
\label{holonomic criterion}
Let $\phi\in \End_{\kk}(D)$ be a endomorphism of $D$. Assume there is a constant $C$ such that $\phi(\scr{F}_i)\subseteq \scr{F}_{Ci}$ for all $i\geq 0$, then the restriction of scalar via $\phi$ preserves holonomicity; that is: if $M$ is holonomic. then $\prescript{\phi}{}M$ is also holonomic.
\end{proposition}

\subsection{Hasse-Schmidt derivations.} We recall the notion of Hasse-Schmidt derivations and we follow \cite[\S27]{Matsumura} closely. Let $\kk$ be a field and $A,B$ be $\kk$-algebras. A \emph{Hasse-Schmidt derivation of length $m$ from $A$ to $B$} is a sequence $\{d_0,d_1,\dots,d_m\}$ of $\kk$-linear maps $A\to B$ such that
\begin{enumerate}
\item $d_0:A\to B$ is a $\kk$-algebra morphism; and
\item for $1\leq i\leq m$ and all $x,y\in A$,
\[d_i(xy)=\sum_{\substack{a,b\geq 0\\a+b=i}} d_a(x)d_b(y).\]
\end{enumerate}
Equivalently, a sequence $\{d_0,d_1,\dots,d_m\}$ of $\kk$-linear maps $A\to B$ is a Hasse-Schmidt derivation if the map $h_m:A\to B_m:=B[[y]]/(y^{m+1})$ defined by $h_m(a)=\sum_{i=0}^md_i(a)y^i$ is a $\kk$-algebra morphism and $d_0:A\to B$ is a $\kk$-algebra morphism. If $\{d_0,d_1,\dots\}$ is a sequence of $\kk$-linear maps $A\to B$ such that $\{d_0,d_1,\dots,d_m\}$ is a Hasse-Schmidt derivation of length $m$ for every positive integer $m$, then the map $h:A\to B[[y]]$ defined by $h(a)=\sum_{i=0}^{\infty} d_i(a)y^i$ is a ring morphism.

If $A=B$, then $\{d_0,d_1,\dots,d_m\}$ is called a Hasse-Schmidt derivation of length $m$ on $A$.

\begin{remark}
\label{image of Keller is HS}
Let $\kk$ be a field and $R=\kk[x_1,\dots,x_n]$. Fix $1\leq j\leq n$ and set $d_i:=\partial^{[i]}_j$ for $i\geq 1$. Then one can check the sequence $\{\id_R, d_1,\dots, d_m\}$ is a Hasse-Schmidt derivation of length $m$ on $R$ for each positive integer $m$. Let $\varphi\in \End_{\kk}(R)$ be a Keller map. Since $\varphi$ is \'{e}tale, it follows from \cite[Theorem 27.2]{Matsumura} that $\{\id_R, \partial_j,\partial_j^{[2]},\dots, \partial_j^{[m]}\}$ admits a unique extension $\{\id_R, d'_1,\dots, d'_m\}$ on $R$ such that $\{\id_R, d'_1,\dots, d'_m\}$ is a Hasse-Schmidt derivation of length $m$ on $R$ and $d'_i(\varphi(r))=\varphi(d_i(r))$. If we denote the unique extension of $\varphi$ in $\End_{\kk}(D)$ by $\Phi$, then $d'_i=\Phi(d_i)$ (by uniqueness). That is, for each $j$, the sequence 
\[\{\id_R, \Phi(\partial_j),\Phi(\partial_j^{[2]}),\dots,\Phi(\partial^{[m]}_j)\}\] 
is a Hasse-Schmidt derivation of length $m$ on R. In particular, $\Phi(\partial^{[m]}_j)$ satisfies ($m\geq 1$)
\[\Phi(\partial^{[m]}_j)(r_1r_2)=\sum_{\substack{a,b\geq 0\\a+b=m}} \Phi(\partial^{[a]}_j)(r_1)\Phi(\partial^{[b]}_j)(r_2)\]
for all $r_1,r_2\in R$. Consequently $h:R\to R[[y]]$ defined via
\[h(r)=\sum_{t=0}^{\infty} \Phi(\partial^{[t]}_j)(r)y^t\]
is a ring morphism for each $j=1,\dots,n$. 
\end{remark}

\begin{theorem}
\label{thm: psi preserve holonomic}
Let $R=\kk[x_1,\dots,x_n]$ with $\ch(\kk)=p>0$ and let $\varphi\in \End(R)$ be a Keller map. Let $\Phi$ denote its unique extension to $\End(D)$. If $M$ is a holonomic $D$-module, then so is $\prescript{\Phi}{}M$.

In particular, $\prescript{\Phi}{}R$ is holonomic and hence has finite length in the category of $D$-modules.
\end{theorem}
\begin{proof}
Let $d:=\max\{\deg(\varphi_1),\dots,\deg(\varphi_n)\}$. We claim that $\Phi(\scr{F}_i)\subseteq \scr{F}_{(2nd+1)i}$ for all $i\geq 1$. Then our theorem will follow from Proposition \ref{holonomic criterion}.

We will start by analyzing $\Phi(\partial^{[b]}_i)$ for all $b\geq 1$ and each $i=1,\dots,n$. (This is one of the major differences between characteristic 0 and characteristic $p$: in characteristic 0, it suffices to consider $\Phi(\partial_i)$ for $i=1,\dots,n$; while in characteristic $p$, one must consider all $\Phi(\partial^{[p^e]}_i)$ for $e\geq 0$ as $\partial^{[p^e]}_i$ ($e\geq 0$) are algebraically independent.) 

Since $\Phi(\partial_i^{[b]})$ (assuming $b\geq 1$) still has order (as a differential operator) $\leq b$, we can write
\begin{equation}
\label{equ: coefficients r}
\Phi(\partial_i^{[b]})=\sum_{\substack{b_1,\dots,b_n\geq 0\\ b_1+\cdots+b_n\leq b}} r_{b_1,\dots,b_n}\partial^{[b_1]}_1\cdots \partial^{[b_n]}_n
\end{equation}
for some $r_{b_1,\dots,b_n}\in R$. As $\partial^{[b_1]}_1\cdots \partial^{[b_n]}_n\in \scr{F}_{\sum_ib_i}\subseteq \scr{F}_b$, it is crucial to bound the degree of $r_{b_1,\dots,b_n}$. One observation is that 
\[r_{0,\dots,0}=0\] 
since $\Phi(\partial^{[b]}_i)(\varphi^{p^e}_i)=0$ for $e\gg 0$ and hence $r_{0,\dots,0}\varphi^{p^e}_i=0$.

To bound the degree of coefficients $r_{b_1,\dots,b_n}$, we would like to have a more explicit description of them. To ease notation, we fix $i$ (the argument for $\Phi(\partial^{[b]}_i)$ works identically for all $\Phi(\partial^{[b]}_j)$) and set
\[\partial_b:=\partial^{[b]}_i\quad {\rm and}\quad \tilde{\partial}_{b}:=\Phi(\partial^{[b]}_i)\quad \forall\ b\geq 1.\]

Since $\{\id_R, \tp,\tp_{2},\dots\}$ is a Hasse-Schmidt derivation (Remark \ref{image of Keller is HS}), the map $h:R\to R[[y]]$ defined by $h(r)=\sum_{t=0}^{\infty}\tp_{t}(r)y^t$ is a $\kk$-algebra morphism. For each $x_j$, we have 
\[h(x_j)=x_j+\tp(x_j)y+\tp_{2}(x_j)y^2+\cdots.\]
Set $z_j(y):=\sum_{t=1}^{\infty}\tp_[{t}(x_j)y^t\in R[[y]]$. Then $h(x_j)=x_j+z_j(y)$. Since $h$ is a $\kk$-algebra morphism, for each $f\in R$, we have
\[h(f)=f(h(x_1),\dots,h(x_n))=f(x_1+z_1(y),\dots,x_n+z_n(y)).\]
Now consider the multi-variable Hasse-Taylor expansion of $f$: 
\[f(x_1+z_1(y),\dots,x_n+z_n(y))=\sum_{b_1,\dots,b_n\geq 0} (\partial^{[b_1]}_1\cdots \partial^{[b_n]}_n)(f)z_1(y)^{b_1}\cdots z_n(y)^{b_n}\]

For each formal power series $g(y)\in R[[y]]$, we will denote by $c_{y^b}(g)$ the coefficient of $y^b$ in $g$. Then it follows that for every $f\in R$
\[\tp_{b}(f)=\sum_{\substack{b_1,\dots,b_n\geq 0\\ b_1+\cdots+b_n\leq b}}c_{y^b}(z_1(y)^{b_1}\cdots z_n(y)^{b_n})(\partial^{[b_1]}_1\cdots \partial^{[b_n]}_n)(f);\]
that is
\begin{equation}
\label{equ: coefficient}
\tp_{b}=\sum_{\substack{b_1,\dots,b_n\geq 0\\ b_1+\cdots+b_n\leq b}}c_{y^b}(z_1(y)^{b_1}\cdots z_n(y)^{b_n})\partial^{[b_1]}_1\cdots \partial^{[b_n]}_n.
\end{equation}
Put it differently, the coefficient $r_{b_1,\dots,b_n}$ in (\ref{equ: coefficients r}) is precisely $c_{y^b}(z_1(y)^{b_1}\cdots z_n(y)^{b_n})$: the coefficient of $y^b$ in $z_1(y)^{b_1}\cdots z_n(y)^{b_n}$. (To clarify, the coefficient $c_{y^b}(z_1(y)^{b_1}\cdots z_n(y)^{b_n})$ is a polynomial in $R$. By $\deg(c_{y^b}(z_1(y)^{b_1}\cdots z_n(y)^{b_n}))$ we mean the degree as defined in $R$. )

{\it Claim.} $\deg(c_{y^b}(z_1(y)^{b_1}\cdots z_n(y)^{b_n}))\leq (2nd)b-n(d-1)$ for each $b\geq 1$.

We will prove this claim by induction on $b$. We start with the case when $b=1$. In this case, the coefficients $c_{y}(z_1(y)^{b_1}\cdots z_n(y)^{b_n})$ (with $\sum_ib_i\leq b=1$) can be derived explicitly from Jacobian matrix $J_{\varphi}$. It follows from $\Phi(\partial_i)(\varphi_j)=\delta_{ij}$ that
\begin{equation}
\label{image of derivatives}
\begin{pmatrix}\Phi(\frac{\partial}{\partial x_1})\\ \vdots \\ \Phi(\frac{\partial}{\partial x_n}) \end{pmatrix}=J^{-1}_{\varphi}\begin{pmatrix}\frac{\partial}{\partial x_1}\\ \vdots \\ \frac{\partial}{\partial x_n}\end{pmatrix}
\end{equation}
where $J^{-1}_{\varphi}$ can be computed as its adjoint matrix divided by $\det(J_{\varphi})$. Since each entry in $J_{\varphi}$ ash degree $\leq d-1$, the entries in the adjoint matrix have degree $\leq (n-1)(d-1)$. Consequently (with $\sum_ib_i\leq b=1$)
\[\deg(c_{y}(z_1(y)^{b_1}\cdots z_n(y)^{b_n}))\leq (n-1)(d-1)\leq 2nd-n(d-1).\]
This proves the claim when $b=1$.

Assume that we have proved our Claim for all $b\leq m$ for an integer $m\geq 1$ and we wish to our claim for $m+1$. 

If $m+1$ is not a power of $p$, write $m+1=\sum_{j=0}^t m_jp^j$ with $0\leq m_j\leq p-1$. Since $m+1$ is not a power of $p$, it follows that $p^j<m+1$ for each $j$ in this expansion. Hence
\[\tp_{m+1}=\Phi(\prod_{j=0}^t (\partial_{p^j})^{m_j})=\prod_{j=0}^t\tp_{p^j}^{m_j}.\] 
Thus the coefficients in $\tp_{m+1})$ is a $\kk$-linear combination of the products of coefficients in $\tp_{p^j})$. Since $p^j<m+1$, by induction the coefficients in $\tp_{p^j}$ satisfy the degree bound as in our Claim. Then one can check that each coefficient in $\tp_{m+1}$ satisfies the degree bound as in our Claim.

Now we consider the case when $m+1=p^e$ for some $e\geq 1$. To this end, we split the right-hand side of (\ref{equ: coefficient}) (when $b=p^e$) into three pieces according to the order of differential operators: 
\begin{enumerate}
\item the differential operators with the highest order $p^e$: 
\[H:=\sum_{\substack{b_1,\dots,b_n\geq 0\\ b_1+\cdots+b_n=p^e}}c_{y^{p^e}}(z_1(y)^{b_1}\cdots z_n(y)^{b_n})\partial^{[b_1]}_1\cdots \partial^{[b_n]}_n\]

\item the differential operators of order between $2$ and $p^e-1$:
\[G:=\sum_{\substack{b_1,\dots,b_n\geq 0\\ 2\leq b_1+\cdots+b_n\leq p^e-1}}c_{y^{p^e}}(z_1(y)^{b_1}\cdots z_n(y)^{b_n})\partial^{[b_1]}_1\cdots \partial^{[b_n]}_n\]

\item the differential operators with the lowest order $1$: 
\[L:=\sum_{\substack{b_1,\dots,b_n\geq 0\\ b_1+\cdots+b_n=1}}c_{y^{p^e}}(z_1(y)^{b_1}\cdots z_n(y)^{b_n})\partial^{[b_1]}_1\cdots \partial^{[b_n]}_n\]

\end{enumerate}
It is clear that $\tp_{p^e}=H+G+L$. Our approaches to $H$, $G$ and $L$ will be different.

We will start with $H$. Recall that 
\[z_i(y)=\tp_{1}(x_i)y+\tp_{2}(x_i)y^2+\cdots+\tp_{j}(x_i)y^j+\cdots=y(\tp_{1}(x_i)+\tp_{2}(x_i)y+\cdots)\]
Hence the coefficient of $y^{p^e}$ in $z_1(y)^{b_1}\cdots z_n(y)^{b_n}$ with $\sum_ib_i=p^e$ is precisely $\tp_{1}(x_1)^{b_1}\cdots \tp_{1}(x_n)^{b_n}$. It follows from the explicit description of $\tp_{1}=\Phi(\partial)$ in (\ref{image of derivatives}) that $\deg(\tp_{1}(x_i))\leq (n-1)(d-1)$ for each $i$. Consequently (with $\sum_ib_i=p^e$)
\begin{align*}
\deg(c_{p^e}(z_1(y)^{b_1}\cdots z_n(y)^{b_n}))&=\deg(\tilde{d}_1(x_1)^{b_1}\cdots \tilde{d}_1(x_n)^{b_n})\\
&\leq \sum_i(n-1)(d-1)b_i=(n-1)(d-1)p^e\leq 2ndp^e-n(d-1)
\end{align*}
This proves the degree bound on $H$.

Next we treat $G$. Since 
\[z_1(y)^{b_1}\cdots z_n(y)^{b_n}=y^{\sum_ib_i}(\tp_{1}(x_1)+\tp_{2}(x_1)y+\cdots)^{b_1}\cdots (\tp_{1}(x_n)+\tp_{2}(x_n)y+\cdots)^{b_n}\]
the coefficient of $y^{p^e}$ in the expansion of $z_1(y)^{b_1}\cdots z_n(y)^{b_n}$ is a $\kk$-linear combination of products of the form 
\[\tp_1(x_1)^{a_{11}}\cdots \tp_{\ell_1}(x_1)^{a_{1\ell_n}}\cdots \tp_1(x_n)^{a_{n1}}\cdots \tp_{\ell_n}(x_n)^{a_{n\ell_n}}\]
with $\sum_{i=1}^n\sum_{j=1}^{\ell_i}a_{ij}j=p^e$. Since $2\leq \sum_ib_i\leq p^e-1$, the indices $\ell_j$ satisfies $\ell_j<p^e$. By induction the degree of the coefficient of $\tp_j$ is $\leq 2ndj-(n-1)(d-1)$ for all $j<p^e$. Hence 
\[\deg(\tp_i(x_j))\leq 2ndi-(n-1)(d-1)\]
 for each $1\leq i\leq \ell_i$ and $i=1,\dots,n$. Consequently
 \begin{align*}
 \deg(\tp_1(x_1)^{a_{11}}\cdots \tp_{\ell_1}(x_1)^{a_{1\ell_1}}\cdots \tp_1(x_n)^{a_{n1}}\cdots \tp_{\ell_n}(x_n)^{a_{n\ell_n}})&\leq 2nd(\sum_ib_i)-n(d-1)(\sum_{i=1}^n\sum_{j=1}^{\ell_i} a_{ij})\\
 &\leq 2ndp^e-n(d-1).
 \end{align*}
This proves that for all indices $(b_1,\dots,b_n)$ with $2\leq \sum_ib_i\leq p^e-1$
\[\deg(c_{p^e}(z_1(y)^{b_1}\cdots z_n(y)^{b_n}))\leq 2ndp^e-n(d-1)\]
 which completes our argument for $G$.
 
 Finally, we consider $L$. Since $L$ corresponds to the multi-induces $(b_1,\dots,b_n)$ with $\sum_ib_i=1$, it follows that
\[\tp_{p^e}(x_j)=(L+G+H)(x_j)=L(x_j)\]
since each differential operator in $G$ and $H$ has order $\geq 2$. Hence
\[L=\tp_{p^e}(x_1)\partial_1+\cdots +\tp_{p^e}(x_n)\partial_n.\]
It remains to prove the bound of degree of $\tp_{p^e}(x_j)$ as in our Claim. 

Since $\partial^{[p^e]}_i(x_j)=0$ for all $i,j=1,\dots,n$, we have $\Phi(\partial^{[p^e]}_i)(\varphi(x_j))=0$ for all $i,j=1,\dots,n$. We specialize this to $\tp_{p^e}=\Phi(\partial^{[p^e]}_i)$ to get $(L+G+H)(\varphi_j)=0$ for each $j=1,\dots,n$. That is, $L(\varphi_j)=-(G+H)(\varphi_j)$ for each $j=1,\dots,n$. Or equivalently, in terms of the Jacobian matrix $J_{\varphi}$, 
\[J_{\varphi}\begin{pmatrix}\tp_{p^e}(x_1)\\ \vdots\\ \tp_{p^e}(x_n) \end{pmatrix}=-\begin{pmatrix}(G+H)\varphi_1\\ \vdots\\ (G+H)\varphi_n \end{pmatrix}\]
Thus,
\[\begin{pmatrix}\tp_{p^e}(x_1)\\ \vdots\\ \tp_{p^e}(x_n) \end{pmatrix}=-J^{-1}_{\varphi}\begin{pmatrix}(G+H)\varphi_1\\ \vdots\\ (G+H)\varphi_n \end{pmatrix}\]
Recall that $d=\max\{\deg(\varphi_1),\dots,\deg(\varphi_n)\}$. Then $\partial^{[b_1]}_1\cdots \partial^{[b_n]}_n(\varphi_j)=0$ whenever $\sum_ib_i>d$. It follows that 
\[(G+H)(\varphi_j)=\left(\sum_{2\leq b_1+\cdots+b_n\leq d}c_{y^{p^e}}(z_1(y)^{b_1}\cdots z_n(y)^{b_n})\partial^{[b_1]}_1\cdots \partial^{[b_n]}_n \right)(\varphi_j)\]
Recall that each entry in $J^{-1}_{\varphi}$ has degree at most $(n-1)(d-1)$. Hence it follows from the argument as for degree bounds on coefficients in $G$ and $H$ that
\[\deg(\tp_{p^e}(x_j))\leq (n-1)(d-1)+2ndp^e-2n(d-1)+(d-\sum_ib_i).\]
Here is why $-2n(d-1)$ appears: if there is a $b_i$ such that $b_i=b\geq 2$, then $\deg(\tp_{p^e}(x_j))\leq (n-1)(d-1)+2ndp^e-b_in(d-1)+(d-\sum_ib_i)$; otherwise, there will be at least two positive $b_i$'s, then each of $z_i(y)^{b_i}$ contributes $-n(d-1)$ to the degree bound. Therefore
\[\deg(\tp_{p^e}(x_j))\leq 2ndp^e-n(d-1).\]
This completes the proof of our claim.

Now we will prove that $\Phi(\scr{F}_b)\subseteq \scr{F}_{(2nd+1)b}$ for all $b\geq 1$ using our Claim.

Observe that since 
\[\Phi(x^{a_1}_1\cdots x^{a_n}_n)=\varphi(x_1)^{a_1}\cdots \varphi(x_n)^{a_n}=\prod_i\varphi_i^{a_i},\] 
it follows that $\Phi(x^{a_1}_1\cdots x^{a_n}_n)\in \scr{F}_{d\sum_ia_i}\subseteq \scr{F}_{(2nd+1)\sum_ia_i}$ for all $a_i\geq 0$ and $i=1,\dots,n$. Hence, to show that $\Phi(\scr{F}_b)\subseteq \scr{F}_{(2nd+1)b}$, it suffices to show $\Phi(\partial^{[b]}_i)\in \scr{F}_{(2nd+1)b}$ for each $1\leq i\leq n$ and all $b\geq 1$ (recall that $\Phi$ is a $\kk$-algebra endomorphism). 

It follows from (\ref{equ: coefficient}) that 
\[\Phi(\partial^{[b]}_i)\in\scr{F}_{(2ndb-n(d-1))+b}\subseteq \scr{F}_{(2nd+1)b}.\]

This finishes the proof of our theorem.
\end{proof}

\section{An extension of Frobenius that does not preserve holonomicity}
\label{counter example}
In characteristic 0, Bavula's celebrated theorem (\cite[Theorem 1.3]{BavulaAnnals}) asserts that every $\kk$-algebra endomorphism of $D$ preserves holonomicity. It is natural ask whether the same conclusion holds in characteristic $p$. In this section, we present an example of an extension of the Frobenius $F:\kk[x]\to \kk[x]$ which does not preserve holonomicity. Hence in characteristic $p$:
\begin{enumerate}
\item if $\varphi\in \End_{\kk}(R)$ is Keller, then it admits a unique extension $\Phi\in \End_{\kk}(D)$ and the restriction of scalar via $\Phi$ preserves holonomicity, or
\item if $\varphi\in \End_{\kk}(R)$ is not Keller, then it may admit an extension $\tv\in \End_{\kk}(D)$ such that the restriction of scalar via $\tv$ does {\it not} preserve holonomicity.
\end{enumerate}

Our example is based on Example \ref{ex: infinite extension of Frob} and inspired by the construction of examples in \cite{KLM12}. 
\begin{example}
\label{example not preserve holo}
Let $R=\FF_p[x]$. For each integer $e\geq 1$, let $\tilde{F}$ and $\alpha_e:D\to D$ be defined as in Example \ref{ex: infinite extension of Frob}. Define $\Phi:D\to D$ as 
\[\Psi:=(\prod_{e\geq 0}\alpha_e)\circ \tilde{F}.\]
As shown in (\ref{alpha on small t}) that for each $\theta\in D$ there exists $e_0$ such that $\Psi(\theta)=(\prod_{e=0}^{e_0}\alpha_e)\circ \tilde{F}(\theta)$. Hence $\Psi(\theta)$ is well-defined for each $\theta\in D$ and one can check that $\Psi$ is a $\FF_p$-algebra endomorphism on $D$.

{\it Claim.} $\prescript{\Psi}{}R$ is not holonomic.

\begin{proof}
Assume otherwise and we will deduce a contradiction. Since $\prescript{\Psi}{}R$ is assumed to be holonomic, it must be cyclic (Remark \ref{holonomic is cyclic}). Let $g\in R$ be a generator. It follows from Lemma \ref{filtration by generator} that
\[\limsup_{i\to \infty}\frac{\dim_{\FF_p}(\scr{F}_i\cdot g)}{i}=\limsup_{i\to \infty}\frac{\dim_{\FF_p}(\Psi(\scr{F}_i)(g))}{i}<\infty.\]

Our goal is to show that there is an infinite sequence of integers $\{i_e\}$ such that 
\[\limsup_{i_e\to \infty}\frac{\dim_{\FF_p}(\scr{F}_{i_e}\cdot g)}{i_e}\]
is unbounded and hence provide the desired contradiction.

Since $g$ is a generator, there exists a $\theta\in D$ such that $\Psi(\theta)(g)=1$. Fix an integer $\ell$ such that $\theta\in \scr{F}_{\ell}$. We claim that we can take the sequence $\{i_e:=2p^e+\ell\}$.

To this end, for each $t\leq e$, we consider $\Psi(\scr{F}_{2p^t})(1)\subseteq \Psi(\scr{F}_{2p^t+\ell})(g)$ since $1\in \Psi(\scr{F}_{\ell})(g)$. For each $0\leq s\leq p^t$, it is clear that $x^s\partial^{[p^t]}\in \scr{F}_{2p^t}\subseteq \scr{F}_{2p^e}$. One can check that $\Psi(x^s\partial^{[p^t]})(1)$ is a polynomial with degree $p^{4t+2}+s+1$. Thus, $\{\Psi(x^s\partial^{[p^t]})(1)\mid 0\leq s\leq p^t\}$ is a set of polynomials with distinct degrees; consequently it is an $\FF_p$-linearly independent set for each $t\leq e$. Moreover, if $t_1<t_2$, then $p^{4t_1+2}+p^{t_1}+1<p^{4t_2+2}$. Hence the degrees of polynomials in $\{\Psi(x^s\partial^{[p^t_1]})(1)\mid 0\leq s\leq p^{t_1}\}$ are strictly less than the degrees of those in $\{\Psi(x^s\partial^{[p^{t_2}]})(1)\mid 0\leq s\leq p^{t_2}\}$. Therefore, 
\[\{\Psi(x^s\partial^{[p^t]})(1)\mid 0\leq s\leq p^t,\ t\leq e\}\]
is an $\FF_p$-linearly independent set of polynomials. One can check that
\[\dim_{\FF_p}(\{\Psi(x^s\partial^{[p^t]})(1)\mid 0\leq s\leq p^t,\ t\leq e\})\sim O(ep^e)\]
Therefore, 
\[\limsup_{e\to\infty}\frac{\dim_{\FF_p}(\{\Psi(x^s\partial^{[p^t]})(1)\mid 0\leq s\leq p^t,\ t\leq e\})}{2p^e}\]
is unbounded. Consequence
\[\limsup_{e\to\infty}\frac{\dim_{\FF_p}(\Psi(\scr{F}_{2p^e+\ell}))}{2p^e+\ell}\geq \limsup_{e\to \infty}\frac{\dim_{\FF_p}(\{\Psi(x^s\partial^{[p^t]})(1)\mid 0\leq s\leq p^t,\ t\leq e\})}{2p^e+\ell}\]
is unbounded (note that $\ell$ is a fixed constant and independent of $e$). This completes the proof of our Claim.
\end{proof}
\end{example}

\section*{Acknowledgements} The author is grateful to Gennady Lyubeznik for discussions on $D$-modules and for his support over the years. The author thanks Vladimir V. Bavula for his valuable comments on a preliminary version of this article which improved the exposition.

\end{document}